\documentclass{article}
\usepackage{filecontents}

\usepackage[english]{babel}

\usepackage[letterpaper,top=2cm,bottom=2cm,left=3cm,right=3cm,marginparwidth=1.75cm]{geometry}

\usepackage[utf8]{inputenc}
\usepackage{amsmath,amssymb,amsthm,verbatim}
\usepackage{graphicx}
\usepackage{hyperref}
\usepackage{multicol} 
\usepackage{amsmath}
\usepackage{tikz}
\usepackage{mathdots}
\usepackage{yhmath}
\usepackage{cancel}
\usepackage{color}
\usepackage{siunitx}
\usepackage{array}
\usepackage{multirow}
\usepackage{amssymb}
\usepackage{gensymb}
\usepackage{tabularx}
\usepackage{stmaryrd}
\usepackage{booktabs}
\usetikzlibrary{fadings}
\usepackage[thinlines]{easytable}

\theoremstyle{plain}
\newtheorem{theorem}{Theorem}[section]

\newtheorem{lemma}[theorem]{Lemma}
\newtheorem{cor}[theorem]{Corollary}

\theoremstyle{definition}
\newtheorem{defn}[theorem]{Definition}

\newtheorem{fact}[theorem]{Fact}
\newtheorem{question}[theorem]{Question}

\usepackage{physics}
\usepackage{amsmath}
\usepackage{tikz}
\usepackage{mathdots}
\usepackage{yhmath}
\usepackage{cancel}
\usepackage{color}
\usepackage{siunitx}
\usepackage{array}
\usepackage{multirow}
\usepackage{amssymb}
\usepackage{gensymb}
\usepackage{tabularx}
\usepackage{extarrows}
\usepackage{booktabs}
\usetikzlibrary{fadings}
\usetikzlibrary{patterns}
\usetikzlibrary{shadows.blur}
\usetikzlibrary{shapes}

\title{\small \textbf{SPLITTING FAMILIES, REAPING FAMILIES AND FAMILIES OF PERMUTATIONS ASSOCIATED WITH ASYMPTOTIC DENSITY}}
\author{\normalsize David Valderrama\footnote{This project was partially supported by grant INV-2024-199-3207, Facultad de Ciencias, Universidad de los Andes. }}
\date{}

\begin{document}

\renewcommand{\abstractname}{\vspace{-\baselineskip}}

\maketitle

\begin{abstract}
\noindent \textit{Abstract.} We investigate several relations between cardinal characteristics of the continuum related with the asymptotic density of the natural numbers and some known cardinal invariants. Specifically, we study the cardinals of the form $\mathfrak{s}_X$, $\mathfrak{r}_X$ and $\mathfrak{dd}_{X,Y}$ introduced in \cite{FARKAS_KLAUSNER_LISCHKA_2023} and \cite{brech2024densitycardinals}, answering some questions raised in these papers. In particular, we prove that $\mathfrak{s}_0=$ cov$(\mathcal{M})$ and $\mathfrak{r}_0=$ non$(\mathcal{M})$. We also show that $\mathfrak{dd}_{\{r\}, \textsf{all}}=\mathfrak{dd}_{\{1/2\}, \textsf{all}}$ for all $r\in (0,1)$, and we provide a proof of Con($\mathfrak{dd}_{(0,1),\{0,1\}}^{\textsf{rel}}<$ non$(\mathcal{N})$) and Con($\mathfrak{dd}_{\textsf{all},\textsf{all}}^{\textsf{rel}}<$ non$(\mathcal{N})$).

\hspace{2mm}


\hspace{1mm}


\end{abstract}

\section{Introduction}

Given a subset $A$ of $\omega$, the \textit{lower density} of $A$ is 
\[
\underline{d}(A):= \liminf_{n \to \infty} \frac{|A\cap n|}{n},
\]
and the \textit{upper density} of $A$ is 
\[
\overline{d}(A):= \limsup_{n \to \infty} \frac{|A\cap n|}{n}
\]
We say that $A$ has density $r\in[0,1]$ if $d(A):=\lim_{n \to \infty} \frac{|A \cap n|}{n}=r$. Lately, many cardinal invariants have been defined using this notion of asymptotic density, particularly to create variants of the splitting or reaping number (see \cite{Brendle_2023}), or to answer the question of how
many permutations are necessary to alter the density of infinite-coinfinite sets (see \cite{brech2024densitycardinals}). In this paper, we provide several new results about these cardinals and its relations with known cardinal characteristics. First, let us recall main definitions. Say $d(A)=$ \textsf{osc} if $\underline{d}(A)< \overline{d}(A)$. Define \textsf{all}:= $[0,1]\cup \{\textsf{osc}\}$. 

\begin{defn}\label{density cardinals}
Assume $X,Y \subseteq$ \textsf{all} are such that $X\neq \emptyset$ and for all $x\in X$ there is $y \in Y$ with $x\neq y$. The \textit{$(X,Y)$-density number} $\mathfrak{dd}_{X,Y}$ is the smallest cardinality of a family $\Pi\subseteq$ Sym$(\omega)$ such that for every infinite-coinfinite set $A\subseteq \omega$ with $d(A)\in X$ there is $\pi \in \Pi$ such that $d(\pi[A])\in Y$ and $d(\pi[A])\neq d(A)$. 
\end{defn}

In \cite{brech2024densitycardinals}, they give many upper and lower bounds for these cardinals and prove that $\mathfrak{dd}_{[0,1],\textsf{all}}=$ non$(\mathcal{M})$, where non$(\mathcal{M})$ stands for the uniformity of the ideal of the meager sets. However, it is still an open problem if one cardinal of the form $\mathfrak{dd}_{X,Y}$ is consistently different from a known cardinal characteristic. We found that there is a relation with other cardinals defined using relative density. 

\begin{defn}
Let $A$ and $B$ be subsets of $\omega$. Define the \textit{lower relative density of $A$ in $B$}
\[
\underline{d}_B(A)= \liminf_{n\to\infty} \frac{|A\cap B \cap n|}{|B\cap n|}
\]
and the \textit{upper relative density of $A$ in $B$}
\[
\overline{d}_B(A)= \limsup_{n\to\infty} \frac{|A\cap B \cap n|}{|B\cap n|}
\]
If $\underline{d}_B(A)=\overline{d}_B(A)$, the common value $d_B(A)$ is the \textit{relative density of $A$ in $B$}. Say $d_B(A)=$ \textsf{osc} if $\underline{d}_B(A)<\overline{d}_B(A)$.
\end{defn}

\begin{defn}\label{Xreaping}
Let $X$ be a nonempty subset of \textsf{all}. The \textit{$X$-reaping number} $\mathfrak{r}_X$ is the smallest cardinality of a family $\mathcal{R} \subseteq [\omega]^\omega$ such that for all $r\in X$ and for all infinite-coinfinite subsets $S$ of $\omega$ there is $R\in \mathcal{R}$ such that $d_R(S)\neq r$. 
\end{defn}

In figure \ref{Main relations}, we present the most important relations between these cardinals established in this paper. Note that $\mathfrak{r}_{\{1/2\}}$ is the $1/2$-reaping number $\mathfrak{r}_{1/2}$ presented in \cite{Brendle_2023}, so to keep this notation, we just write $\mathfrak{r}_{\rho}$ when $X=\{\rho\}$. In any case, Farkas, Klausner and Lischka prove $\mathfrak{r}_\rho=\mathfrak{r}_{1/2}$ for all $\rho\in(0,1)$ (see \cite[Theorem 5.3]{FARKAS_KLAUSNER_LISCHKA_2023}). There is a similar behavior with the density numbers. In Section \ref{section 1}, we prove $\mathfrak{dd}_{\{1/2\},\textsf{all}}=\mathfrak{dd}_{\{r\},\textsf{all}}$ for all $r\in(0,1)$ giving an answer to \cite[Conjecture 40]{brech2024densitycardinals}.

\begin{figure}[h!]
    \centering

\begin{tikzpicture}

\node (COVN) at (0,0) {cov($\mathcal{N}$)};
\node (QO) at (0,1) {$\mathfrak{dd}_{\{1/2\},\textsf{all}}$};
\node (OPO) at (-1,2) {$\mathfrak{dd}_{(0,1),\textsf{all}}$};
\node (RI) at (0,3) {$\mathfrak{r}_{(0,1)}$};
\node (COVSN) at (-1,4) {cov$(\mathcal{SN})$};
\node (OPC) at (1,4) {$\mathfrak{dd}_{(0,1),[0,1]}$};
\node (OPOP) at (1,5) {$\mathfrak{dd}_{(0,1),(0,1)}$};
\node (R) at (2,6) {$\mathfrak{r}$};
\node (OA) at (3,4) {$\mathfrak{dd}_{\{\textsf{osc}\}, \textsf{all}}$};
\node (FR) at (3,3) {$\mathfrak{fr}=\min\{\mathfrak{r},\mathfrak{d}\}$};
\node (RQ) at (1,2) {$\mathfrak{r}_{1/2}$};
\node (RO) at (3,5) {$\mathfrak{r}_{\textsf{osc}}$};

\draw[dashed]  (COVN) -- (QO) ;
\draw[dashed] (QO) -- (OPO);
\draw (RI) -- (COVSN);
\draw (RI) -- (OPC);
\draw[dashed] (OPC) -- (OPOP);
\draw[dashed] (OPOP) -- (R);
\draw[dashed] (OA) -- (RO);
\draw (FR) -- (OA);
\draw[dashed] (OPO) -- (RI);
\draw[dashed] (RQ) -- (RI);
\draw[dashed] (QO) -- (RQ);
\draw[dashed] (RO) -- (R);

\end{tikzpicture}
    
    \caption{ Cardinals grow as one moves up. Straight lines signify that strict inequality is consistent, and dotted lines mean that it is unknown whether the ordering between the cardinals is consistently strict.}
    \label{Main relations}
\end{figure}
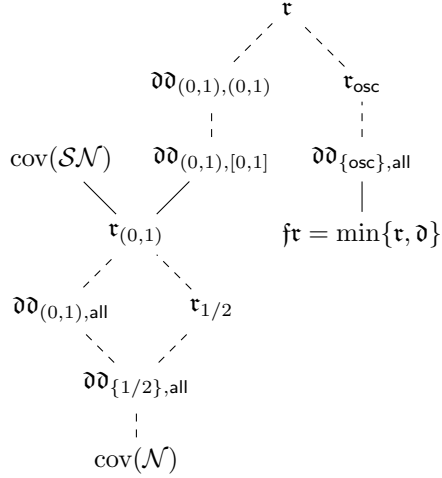

There are dual versions of the cardinals defined in \ref{Xreaping} and \ref{density cardinals}.

\begin{defn}
Let $X$ be a nonempty subset of \textsf{all}. The $X$-splitting number $\mathfrak{s}_X$ is the smallest cardinality of a family $\mathcal{S}$ of infinite-coinfinite subsets of $\omega$ such that for all $R\in[\omega]^\omega$ there exist $S\in \mathcal{S}$ and $r\in X$ such that $d_R(S)=r$. 
\end{defn}

\begin{defn}
Assume $X,Y \subseteq$ \textsf{all} are such that $X\neq \emptyset$ and for all $x\in X$ there is $y \in Y$ with $x\neq y$.  $\mathfrak{dd}_{X,Y}^\perp$ is the smallest cardinality of a family $\mathcal{A}$ of infinite-coinfinite subsets of $\omega$ with $d(A)\in X$ for all $A\in \mathcal{A}$ such that for all $\pi\in$ Sym$(\omega)$ there is $A\in \mathcal{A}$ such that $d(\pi[A])\not\in Y$ or $d(\pi[A])=d(A)$. 
\end{defn}

Some of the results presented in this paper also have a dual version. However, we do not know if $\mathfrak{dd}_{\{r\},\textsf{all}}^\perp= \mathfrak{dd}_{\{1/2\},\textsf{all}}^\perp$ for all $r\in(0,1)$. In Section \ref{section 1}, we also prove con$(\mathcal{N})\leq \mathfrak{r}_\rho$ and $\mathfrak{s}_\rho\leq$ non$(\mathcal{N})$ for all $\rho\in[0,1]$, and we show that $\mathfrak{s}_0=$ cov$(\mathcal{M})$ and $\mathfrak{r}_0=$ non$(\mathcal{M})$. These results answer \cite[Question E]{FARKAS_KLAUSNER_LISCHKA_2023} and \cite[Question F]{FARKAS_KLAUSNER_LISCHKA_2023}. At last, we study a variant of \ref{density cardinals}, which is also introduced in \cite{brech2024densitycardinals}, defined using the notion of relative density instead of asymptotic density.  

\begin{defn}
Assume $X,Y\subseteq$ all are such that $X\neq \emptyset$ and for all $x\in X$ there is $y\in Y$ with $y\neq x$. The $(X,Y)$-relative density number $\mathfrak{dd}_{X,Y}^{\textsf{rel}}$ is the smallest cardinality of a family $\Pi \subseteq$ Sym$(\omega)$ such that for every $A\subseteq B \subseteq \omega$ with $A$ and $B\setminus A$ both infinite and with $d_B(A) \in X$ there is $\pi \in \Pi$ such that $d_{\pi[B]}(\pi[A])\in Y$ and $d_B(A)\neq d_{\pi[B]}(\pi[A])$.
\end{defn}

It is not known whether there are upper bounds for the numbers $\mathfrak{dd}_{(0,1),\{0,1\}}^{\textsf{rel}}$, $\mathfrak{dd}_{(0,1),(0,1)}^{\textsf{rel}}$, and $\mathfrak{dd}_{\textsf{all},\textsf{all}}^{\textsf{rel}}$ aside from the continuum. In Section \ref{section 2}, we prove that $\mathfrak{dd}_{(0,1),\{0,1\}}^{\textsf{rel}}$ and $\mathfrak{dd}_{\textsf{all},\textsf{all}}^{\textsf{rel}}$ are consistently different from the continuum. Specifically, we show the consistency of $\mathfrak{dd}_{(0,1)\{0,1\}}^{\textsf{rel}}<$ non$(\mathcal{N})$ and $\mathfrak{dd}_{\textsf{all},\textsf{all}}^{\textsf{rel}}<$ non$(\mathcal{N})$. This is related to \cite[Question 51]{brech2024densitycardinals} because there is an analogy between $\mathfrak{dd}_{(0,1),\{0,1\}}^{\textsf{rel}}$ and the rearrangement number $\mathfrak{rr}_i$ from \cite{blass2019rearrangementnumber}.

\section{Permutations and relative density}\label{section 1}

Let us see the relation between $\mathfrak{r}_X$ and $\mathfrak{dd}_{X,\textsf{all}}$. Before, we need one more definition.

\begin{defn}
Let $\pi $ be an permutation of $\omega$ and $A$ be a subset of $\omega$. We say that $\pi$ \textit{preserves} $A$ if for all but finitely many elements of $A$ we have $x<y$ if and only if $\pi(x)<\pi(y)$.
\end{defn}

\begin{lemma}\label{permutations-infinite sets}
Let $A$ be an infinite subset of $\omega$. There exists a permutation $\pi \in$ sym$(\omega)$ such that for all $B\in[\omega]^\omega$ we have
\[
\text{if } d(\pi[B])=d(B) \text{ then } d_A(B)=d(B).
\]
\end{lemma}
\begin{proof}
Given an infinite set $A$, consider a permutation $\pi\in$ sym$(\omega)$ such that $d(\pi[A])=1$ and $\pi$ preserves $A$. Let $B\in [\omega]^\omega$ be such that $d(\pi[B])=d(B)$. Note that $d(\pi[B])=d_{\pi[A]}(\pi[B])$ since $d(\pi[\omega\setminus A])=0$. Besides, since $\pi$ preserves $A$, we have $d_{\pi[A]}(\pi[B])=d_{A}(B)$. Therefore, $d_{A}(B)=d_{\pi[A]}(\pi[B])=d(\pi[B])=d(B)$.
\end{proof}

\begin{cor}\label{bisecting-permutation}
For all nonempty $X\subseteq$ \textsf{all} we have $\mathfrak{dd}_{X,\textsf{all}}\leq \mathfrak{r}_{X}$ and $\mathfrak{s}_{X}\leq\mathfrak{dd}_{X,\textsf{all}}^{\perp}$. 
\end{cor}
\begin{proof}
Let $\mathcal{R}_X$ be a family of subsets of $\omega$ witness of $\mathfrak{r}_X$. For every $A\in \mathcal{R}_X$ consider the permutation $\pi_A$ given by \ref{permutations-infinite sets}, and define $\Pi_X:=\{\pi_A \,:\, A\in\mathcal{R}_X\}$. If there exists an infinite-coinfinite set $B\subseteq\omega$ with $d(B)=r\in X$ such that $d(\pi_A[B])=d(B)$ for all $A\in \mathcal{R}_X$, then, by \ref{permutations-infinite sets},  $d_A(B)=r \in X$ for all $A\in \mathcal{R}_X$, which is a contradiction. Therefore, $\mathfrak{dd}_{X,\textsf{all}}\leq |\Pi_X|\leq\mathfrak{r}_{X}$. Now, to see the other inequality, let $\mathcal{S}_X$ be a family of infinite-coinfinite subsets of $\omega$ witness of $\mathfrak{dd}_{X,\textsf{all}}^{\perp}$. Given an infinite set $A\subseteq\omega$, consider the permutation $\pi_A$ given by \ref{permutations-infinite sets}, then there is $B\in \mathcal{S}_X$ with $d(B)=r\in X$ such that $d(\pi_A[B])=d(B)$. This implies that $d_A(B)=r\in X$. Therefore, $\mathfrak{s}_{X}\leq|\mathcal{S}_X|\leq\mathfrak{dd}_{X,\textsf{all}}^{\perp}$.
\end{proof}

Thanks to the previous result, we can say more about $\mathfrak{s}_X$ and $\mathfrak{r}_X$ since density cardinals satisfy the following inequalities. 

\begin{lemma}\label{equality with non and 0-1}
If $X\cap[0,1]\neq\emptyset$, then $\mathfrak{dd}_{X,\textsf{all}}^\perp\leq$ non$(\mathcal{N})$ and cov$(\mathcal{N})\leq \mathfrak{dd}_{X,\textsf{all}}$. Besides, if \textsf{osc} $\not\in X$, and $0\in X$ or $1\in X$, then non$(\mathcal{M})= \mathfrak{dd}_{X,\textsf{all}}$ and $\mathfrak{dd}_{X,\textsf{all}}^\perp =$ cov$(\mathcal{M})$. 
\end{lemma}
\begin{proof}
See \cite[Theorem 16]{brech2024densitycardinals} and \cite[Theorem 18]{brech2024densitycardinals}.    
\end{proof}

\begin{cor}\label{density and nulls}
If $X\cap[0,1]\neq\emptyset$, then cov$(\mathcal{N})\leq \mathfrak{r}_X$ and $\mathfrak{s}_X \leq$ non$(\mathcal{N})$. 
\end{cor}

Corollary \ref{density and nulls} answers the question \cite[Question E]{FARKAS_KLAUSNER_LISCHKA_2023} since, in particular, for all $\rho\in[0,1]$ we have cov$(\mathcal{N})\leq \mathfrak{r}_\rho$ and $\mathfrak{s}_\rho \leq$ non$(\mathcal{N})$. In this paper, they prove that $\mathfrak{r}_\rho\leq$ non$(\mathcal{M})$ and cov$(\mathcal{M}) \leq \mathfrak{s}_\rho$ for all $\rho\in[0,1]$. There is a more general result.

\begin{lemma}\label{inqualitites s-cov}
If $X\subseteq[0,1]$ and $X\neq \emptyset$, then $\mathfrak{r}_X\leq$ non$(\mathcal{M})$ and cov$(\mathcal{M}) \leq \mathfrak{s}_X$. 
\end{lemma}
\begin{proof}
This is basically a consequence of the proof of cov$(\mathcal{M})\leq\mathfrak{s}_{1/2}$ presented in \cite{Brendle_2023}. They prove that for all infinite-coinfinite subsets $S$ of $\omega$ the set $\{R\in[\omega]^\omega \,|\, \limsup_{n \to \infty}\frac{|S\cap R \cap n|}{|R \cap n|}=1 \}$ is a comeagre set. Note that if $\limsup_{n \to \infty}\frac{|(\omega\setminus S) \cap R \cap n|}{|R \cap n|}=1$ then $\liminf_{n \to \infty}\frac{|S\cap R \cap n|}{|R \cap n|}=0$. Therefore, for all infinite-coinfinite subsets $S$ of $\omega$ the set $\{R\in[\omega]^\omega \,|\, \limsup_{n \to \infty}\frac{|S\cap R \cap n|}{|R \cap n|}=1\} \cap \{R\in[\omega]^\omega \,|\, \liminf_{n \to \infty}\frac{|S\cap R \cap n|}{|R \cap n|}=0 \}$ is a comeagre set. In conclusion, the set $\{R\in[\omega]^\omega \,|\, d_R(S)$ oscillates$\}$ is comeagre. For that reason, given a nonempty set $X\subseteq[0,1]$ and a non meager set $\mathcal{R}_X\subseteq[\omega]^\omega$, we have that for all infinite-coinfinite sets $S$ there is $R\in \mathcal{R}_X$ such that $d_R(S)=\textsf{osc}\not\in X$. Therefore $\mathfrak{r}_X\leq$ non$(\mathcal{M})$. By a similar argument, we have cov$(\mathcal{M}) \leq \mathfrak{s}_X$. 
\end{proof}

By \ref{bisecting-permutation}, \ref{equality with non and 0-1} and \ref{inqualitites s-cov}, we get the following corollary answering \cite[Question F]{FARKAS_KLAUSNER_LISCHKA_2023}.

\begin{cor}
For all $X\subseteq[0,1]$ such that $0 \in X$ or $1 \in X$, we have $\mathfrak{s}_X=$ cov$(\mathcal{M})$ and non$(\mathcal{M})=\mathfrak{r}_X$. In particular, $\mathfrak{s}_0=$ cov$(\mathcal{M})$ and non$(\mathcal{M})=\mathfrak{r}_0$.    
\end{cor}

In \cite{FARKAS_KLAUSNER_LISCHKA_2023}, it was also proved that $\mathfrak{r}_{1/2}=\mathfrak{r}_{p}$ for all $p\in (0,1)$. Using an analogous idea, we present a similar result for the density numbers. Like them, we also use the following fact.

\begin{fact}[Non-integer bases]\label{nonintegerbases}
Let $0<b<1$ be a real number. Then every $x>0$ can
be written as $x=\sum_{n=-N}^{\infty}c_n b^n$  where $N\geq 0$ is an integer and $0 \leq c_n < 1/b$ are also
integers for all $n$.  
\end{fact}

\begin{theorem}\label{equalityparameters}
For all $r\in(0,1)$, we have $\mathfrak{dd}_{\{1/2\},\textsf{all}}=\mathfrak{dd}_{\{r\},\textsf{all}}$.
\end{theorem}
\begin{proof}
\begin{description}
    \item[$\mathfrak{dd}_{\{1/2\},\textsf{all}} \leq \mathfrak{dd}_{\{r\},\textsf{all}}$ for all $r \in(0,1)$.] Let $\Pi \subseteq$ sym$(\omega)$ be a family of permutations such that $|\Pi|< \mathfrak{dd}_{\{1/2\},\textsf{all}}$, then there is a set $A\subseteq \omega$ such that $d(A)=1/2$ and $d(\pi[A])=1/2$ for all $\pi \in \Pi$. Consider the increasing enumeration $f: \omega \rightarrow \omega\setminus A$ of $\omega\setminus A$, and for all $\pi\in\Pi$ consider the increasing enumeration $f_\pi: \omega \rightarrow \pi[\omega\setminus A]$ of $\pi[\omega \setminus A]$. Now, define the following family of permutations $\Pi^\prime:= \bigcup_{\pi \in \Pi} \{ f_\pi^{-1} \circ \pi \circ f \}$. Since $|\Pi^\prime|\leq |\Pi| < \mathfrak{dd}_{\{1/2\},\textsf{all}}$, there is a subset $B^\prime$ of $\omega$ such that $d(B^\prime)=1/2$ and $d(\pi^\prime[B^\prime])=1/2$ for all $\pi^\prime \in \Pi^\prime$. Define $B:=f[B^\prime]$. It is easy to check that $d(B)=1/4$ and $d(\pi[B])=1/4$ for all $\pi \in \Pi$. Applying the same argument in $\omega\setminus(A \cup B)$, we can also get a set $C \subseteq \omega\setminus(A \cup B)$ such that $d(C)=1/8$ and $d(\pi[C])=1/8$ for all $\pi \in \Pi$. In fact, using an induction argument, there is a collection $\{A_i\}_{i\in \omega}$ of subsets of $\omega$ such that $A_i \cap A_j = \emptyset$ for all $1 \leq i < j < \omega$ and $d(A_i)=d(\pi[A_i])=\frac{1}{2^i}$ for all $\pi\in \Pi$. Now, given a $r\in[0,1]$, let $g \in 2^\omega$ be such that $r=\sum_{i \in \omega} \frac{g(i)}{2^i}$, that is, the representation of $r$ in base $2$. Define the following sets
    \[
    A_r:= \bigcup_{i \in g^{-1}[\{1\}]} A_i \, \, \, \, \, A_{r-1}:=\bigcup_{i \in g^{-1}[\{0\}]} A_i
    \]
    Note that $r \leq \underline{d}(A_r)$, and $1-r \leq \underline{d}(A_{r-1})$. Since $A_r \cap A_{r-1}= \emptyset$, we have $d(A_r)=r$. Using a similar argument, we have $d(\pi[A_r])=r$ for all $\pi\in\Pi$.

    \item[$\mathfrak{dd}_{\{r\},\textsf{all}} \leq \mathfrak{dd}_{\{1/2\},\textsf{all}}$ for all $r \in (1/3,1/2)$.]  Let $\Pi \subseteq$ sym$(\omega)$ be a family of permutations such that $|\Pi|< \mathfrak{dd}_{\{r\},\textsf{all}}$, then there is a set $A\subseteq \omega$ such that $d(A)=r$ and $d(\pi[A])=r$ for all $\pi \in \Pi$. Using a similar argument to the previous case, there is a pairwise disjoint collection $\{A_i\}_{i\in \omega}$ of subsets of $\omega$ such that $d(A_i)=d(\pi[A_i])=r(1-r)^i$ for all $\pi\in \Pi$. Now, since $1/3 <r < 1/2$, then $\frac{1}{2r}<\frac{1}{1-r}<2$, so, by \ref{nonintegerbases}, there is $g\in 2^\omega$ such that $\frac{1}{2r}=\sum_{i\in \omega} g(i)(1-r)^i$, that is, the representation of $\frac{1}{2r}$ in base $\frac{1}{1-r}$. Define the following sets
    \[
    E:= \bigcup_{i \in g^{-1}[\{1\}]} A_i \, \, \, \, \, F:=\bigcup_{i \in g^{-1}[\{0\}]} A_i
    \]
    Note that $\underline{d}(E)\geq \sum_{i \in g^{-1}[\{1\}]}d(A_i)=\sum_{i \in g^{-1}[\{1\}]}r(1-r)^i=r\sum_{i\in \omega} g(i)(1-r)^i=r \frac{1}{2r}=1/2$. On the other hand, $\underline{d}(F) \geq \sum_{i \in g^{-1}[\{0\}]}d(A_i) =\sum_{i \in \omega}d(A_i) - \sum_{i \in g^{-1}[\{1\}]}d(A_i)= 1-1/2=1/2$. Since $E \cap F =\emptyset$ we have $d(E)=1/2$. Using a similar argument, we have $d(\pi[E])=1/2$ for all $\pi\in\Pi$.

    \item[$\mathfrak{dd}_{\{r\},\textsf{all}} \leq \mathfrak{dd}_{\{1/2\},\textsf{all}}$ for all $0<r\leq 1/3$.] It is enough to find $s\in (1/3,2/3)$ such that $\mathfrak{dd}_{\{r\},\textsf{all}} \leq \mathfrak{dd}_{\{s\},\textsf{all}}$ since $\mathfrak{dd}_{\{s\},\textsf{all}}\leq \mathfrak{dd}_{\{1/2\},\textsf{all}}$ by the previous case.  Let $\Pi \subseteq$ sym$(\omega)$ be a family of permutations such that $|\Pi|< \mathfrak{dd}_{\{r\},\textsf{all}}$, then there is a set $A\subseteq \omega$ such that $d(A)=r$ and $d(\pi[A])=r$ for all $\pi \in \Pi$. Using an argument similar to the previous cases, we can get a set $B\subseteq \omega$ such that $A \cap B = \emptyset$ and $d(B)=d(\pi[B])=r(1-r)=r-r^2$ for all $\pi \in \Pi$. In particular, $d(A\cup B)=d(\pi[A \cup B]))=2r-r^2$. Therefore, $\mathfrak{dd}_{\{r\},\textsf{all}} \leq \mathfrak{dd}_{\{2r-r^2\},all}$. Since $f(x)=2x-x^2$ is an increasing function in $(0,1)$, and $f(1/3)<2/3$, there is $n\in \omega$ such that $s:=f^{(n)}(r)\in (1/3, 2/3)$. 
\end{description}
\end{proof}

\begin{cor}
Let $X$ be a nonempty subset of $(0,1)$ such that $|X|\leq \aleph_1$, then $\mathfrak{dd}_{X,\textsf{all}}=\mathfrak{dd}_{\{1/2\},\textsf{all}}$. In particular, $\mathfrak{dd}_{\mathbb{Q}\cap(0,1),\textsf{all}}=\mathfrak{dd}_{\{1/2\},\textsf{all}}$.   
\end{cor}
\begin{proof}
Since $X$ is nonempty, there is $r\in X$, so $\mathfrak{dd}_{\{1/2\},\textsf{all
}}=\mathfrak{dd}_{\{r\},\textsf{all}}\leq\mathfrak{dd}_{X,\textsf{all}}$. For the other inequality, for all $r\in X$ consider a family $\Pi_r\subseteq$ Sym$(\omega)$ witness of $\mathfrak{dd}_{\{r\},\textsf{all}}$. Note that $\mathfrak{dd}_{X,\textsf{all}} \leq |\bigcup_{r\in X} \Pi_r|=\max\{\aleph_1,\mathfrak{dd}_{\{1/2\},\textsf{all}}\}=\mathfrak{dd}_{\{1/2\},\textsf{all}}$. 
\end{proof}

Now we are going to see that $\mathfrak{dd}_{X,[0,1]}$ is an upper bound of $\mathfrak{r}_X$ for all $X \subseteq [0,1]$. We need a combinatorial property which is quite similar to \cite[Lemma 4]{blass2019rearrangementnumber}.  

\begin{lemma}\label{change to osc-pre}
Let $\pi$ be a permutation of $\omega$. There is a permutation $\tau \in $ sym($\omega$) such that   
\begin{itemize}
    \item $\tau[n]=n$ for infinitely many $n\in\omega$.
    \item $\tau[n]=\pi[n]$ for infinitely many $n\in\omega$.
    \item There is an $A\in[\omega]^\omega$ such that $\tau$ preserves $A$ and $\omega\setminus A$. 
\end{itemize}
\end{lemma}
\begin{proof}
We will define, by induction on $k$, injective functions $\tau_k:n_k\rightarrow \omega$ to serve as initial segments of $\tau$, and finite sets $A_k$, which will be the finite segments $A\cap[0,n_k]$, such that for all $x,y \in A_k$, and for all $x,y\in n_k \setminus A_k$, we have $x<y$ if and only if $\tau_k(x)<\tau_k(y)$.  First, define $\tau_0$ as the empty function and $A_0=\emptyset$. Then, suppose $\tau_k: n_k \rightarrow \omega$ and $A_k$ are already constructed for $k$ even number. Besides, we can suppose that $\tau_k[n_k]=n_k$. To construct $\tau_{k+1}$, define $n_{k+1}:=\max\pi^{-1}[n_k]+1$ and $B:=\pi[n_{k+1}]\setminus n_k$. Note that $|B|=n_{k+1}- n_k$. Then, consider the bijection $g:n_{k+1} \setminus n_k \rightarrow B$ that preserves the order of $n_{k+1} \setminus n_k$. Set $\tau_{k+1}:=\tau_k \cup g$ and $A_{k+1}=A_k \cup n_{k+1} \setminus n_k$. Note that we have $\tau_{k+1}[n_{k+1}]=\pi[n_{k+1}]$. Now, to construct $\tau_{k+2}$ define $n_{k+2}:=\max\tau_{k+1}[n_{k+1}]+1$ (if $n_{k+2}=n_{k+1}$ reset $n_{k+2}$ to $n_{k+1}+1$) and $C:=n_{k+2}\setminus \tau_{k+1}[n_{k+1}]$. Note that $|C|=n_{k+2}-n_{k+1}$. Then, consider the bijection $h: n_{k+2}\setminus n_{k+1} \rightarrow C$ that preserves the order of $n_{k+2} \setminus n_{k+1}$. Set $\tau_{k+2}:=\tau_{k+1} \cup h$ and $A_{k+2}=A_{k+1}$. Note that we have $\tau_{k+2}[n_{k+2}]=n_{k+2}$. Since $n_{k+2}\setminus A_{k+2} = n_{k+2} \setminus n_{k+1} \cup n_k \setminus A_k$, the order of this set is preserved by $\tau_{k+2}$ by construction of $h$ and the hypothesis of induction. Finally define $\tau:=\bigcup_{k<\omega} \tau_k$ and $A:=\bigcup_{k<\omega} A_k$. 
\end{proof}

\begin{lemma}\label{relation permutation-2set}
Let $\pi$ be a permutation of $\omega$ and $\rho\in(0,1)$. Then, there exists an infinite set $A\subseteq \omega$ such that for all $B\in[\omega]^\omega$ we have
\[
\text{if } d_{\omega\setminus A}(B)=d_A(B)=\rho \text{ then, } d(\pi[B])=d(B) \text{ or } d(\pi[B]) \text{ oscillates.}  
\]
\end{lemma}
\begin{proof}
Let $\pi$ be a permutation of $\omega$.  By \ref{change to osc-pre}, consider a permutation $\tau$ and a set $A$ such that $\tau^{-{1}}[n]=n$ for infinitely many $n\in \omega$, $\tau^{-{1}}[n]=\pi^{-{1}}[n]$ for infinitely many $n\in \omega$, and $\tau$ preserves $A$ and $\omega\setminus A$. Let $B\in[\omega]^\omega$ be such that $d_{\omega\setminus A}(B)=d_A(B)=\rho$. Since $\tau$ preserves $A$ and $\omega\setminus A$, we have $d_{\tau[A]}(\tau[B])=d_A(B)=\rho$ and $d_{\tau[\omega\setminus A]}(\tau[B])=d_{\omega\setminus A}(B)=\rho$. It is easy to check that this implies $d(\tau[B])=\rho=d(B)$. Recall that if $d(B)\neq d(\pi[B])$ and $d(\pi[B])\in[0,1]$, by the others properties of $\tau$, then $d(\tau[B])$ oscillates. Therefore, $d(\pi[B])=d(B)$ or $d(\pi[B])$ oscillates. 
\end{proof}

\begin{cor}
For all $X\subseteq [0,1]$ with $X\neq\emptyset$ we have $\mathfrak{r}_{X}\leq \mathfrak{dd}_{X,[0,1]}$. In particular, $\mathfrak{r}_{(0,1)}\leq \mathfrak{dd}_{(0,1),[0,1]}$.   
\end{cor}
\begin{proof}
Let $\Pi \subseteq$ sym$(\omega)$ be a witness of $\mathfrak{dd}_{X,[0,1]}$. For each $\pi\in\Pi$ denote as $A_\pi$ the set given by \ref{relation permutation-2set}. Consider the collection $\mathfrak{R}:= \bigcup_{\pi \in \Pi} \{A_\pi,\omega\setminus A_\pi\}$. Suppose that there is $\rho\in X$ and $B\subseteq \omega$ such that $d_{\omega\setminus A_\pi}(B)=d_{A_\pi}(B)=\rho$ for all $\pi \in \Pi$. By \ref{relation permutation-2set}, $d(\pi[B])=d(B)$ or $d(\pi[B])$ oscillates for all $\pi \in \Pi$, but $d(B)=\rho$, so this is a contradiction. 
\end{proof}

In \cite[Corollary 2.7]{valderrama2024cardinalinvariantsrelateddensity} we proved that $\mathfrak{r}_{1/2}=\aleph_1$ in the Hechler model. A slight modification of that argument also shows that $\mathfrak{r}_{(0,1)}=\aleph_1$ in the Hechler model. Besides, in \cite[Corollary 21]{brech2024densitycardinals}, it was proved that $\mathfrak{b}\leq\mathfrak{dd}_{\rho,[0,1]}$. Therefore, in the Hechler model $\mathfrak{r}_{(0,1)}$ is strictly less than $\mathfrak{b}$ and $\mathfrak{dd}_{\rho,[0,1]}$ for all $\rho\in(0,1)$. Now, let us see that $\min\{\mathfrak{d},\mathfrak{r}\} \leq \mathfrak{dd}_{\{\textsf{osc}\},\textsf{all}}$.

\begin{defn}
We say that a sequence $\mathcal{A}=(A_n : n \in \omega )$ of finite subsets of $\omega$ is a block sequence if $\max(A_n)<\min(A_{n+1})$. We say that a set $B \in [\omega]^\omega$ \textit{splits} $\mathcal{A}$ if both $\{n \in \omega \, | \, A_n \subseteq B \}$ and $\{n \in \omega \, | \, A_n \cap B = \emptyset \}$ are infinite.
\end{defn}

There are two cardinal characteristics associated with the previous notion. 
\[
\begin{array}{ll}
    \mathfrak{fs}&= \min\{ | \mathcal{F} | \, | \,  \mathcal{F} \subseteq [\omega]^\omega \text{ and every block sequence is split by a member of $\mathcal{F}$} \}   \\
    \mathfrak{fr}&= \min\{ | \mathcal{F} | \, | \,  \mathcal{F} \text{ consists of block sequences and no single $A\in [\omega]^\omega$ splits all members of $\mathcal{F}$} \} 
\end{array}
\]
It is known that $\mathfrak{fs} = \max\{ \mathfrak{b}, \mathfrak{s} \}$ and $\mathfrak{fr} = \min\{ \mathfrak{d}, \mathfrak{r} \}$ (see \cite{Kamburelis1996-KAMS} and \cite{Brendle1998}).

\begin{lemma}\label{permutations-block sequence}
Let $\pi$ be a permutation of $\omega$. Then, there exists a block sequence $\mathcal{A}=(A_n : n \in \omega)$ such that for all $B\in[\omega]^\omega$ we have
\[
\text{ if $B$ splits $\mathcal{A}$ then } d(\pi[B]) \text{ oscillates }
\]
\end{lemma}
\begin{proof}
We are going to construct the block sequence $\mathcal{A}=(A_n : n \in \omega)$ by induction on $\omega$. Set $A_0=\{0\}$. Now suppose we already constructed $A_n$. Define $i_n:=\max\pi[\max A_n +1]+1$, and consider the interval of natural numbers $I:=[i_n,2^{i_n})$. Set $A_{n+1}:=\pi^{-1}[I]$. Now, let $B\in[\omega]^\omega$ be such that $B$ splits $\mathcal{A}$, and let $n\in \omega$. If $A_n \subseteq B$, then 
\[
\frac{|\pi[B]\cap 2^{i_n}|}{2^{i_n}}\geq \frac{2^{i_n}-i_n}{2^{i_n}} = 1 - \frac{i_n}{2^{i_n}}.
\]
In a similar way, if $A_n \cap B =\emptyset$ then
\[
\frac{|\pi[B]\cap 2^{i_n}|}{2^{i_n}}\leq \frac{i_n}{2^{i_n}}
\]
Since there are infinitely many $n\in \omega$ such that $A_n \subseteq B$ and there are infinitely many $n\in \omega$ such that $A_n \cap B =\emptyset$, we conclude that $d(\pi[B])$ oscillates.
\end{proof}

\begin{cor}
$\mathfrak{fr} \leq \mathfrak{dd}_{\{\textsf{osc}\},\textsf{all}}$    
\end{cor}
\begin{proof}
Let $\Pi\subseteq$ sym$(\omega)$ be a witness of $\mathfrak{dd}_{\{\textsf{osc}\},\textsf{all}}$. Without loss of generality, we can suppose that the identity function $id$ belongs to $\Pi$. For all $\pi\in \Pi$ consider the block sequence $\mathcal{A}_\pi$ given by \ref{permutations-block sequence}. If there exists $B\in[\omega]^\omega$ such that $B$ splits $\mathcal{A}_\pi$ for all $\pi\in \Pi$, then by \ref{permutations-block sequence} $d(\pi[B])$ oscillates for all $\pi\in \Pi$. In particular, $d(id[B])=d(B)$ oscillates. Therefore, that set $B$ cannot exist by the definition of $\Pi$.
\end{proof}

\begin{cor}
$\mathfrak{dd}_{\{\textsf{osc}\},\textsf{all}}^\perp \leq \mathfrak{fs}$     
\end{cor}
\begin{proof}
Define $\mathfrak{fs}^\star$ as the least cardinality of a family $\mathcal{F} \subseteq [\omega]^\omega$ such that for all $F\in\mathcal{F}$ we have $d(F)$ oscillates, and every block sequence is split by a member of $\mathcal{F}$. By \ref{permutations-block sequence}, it is easy to check that $\mathfrak{dd}_{\{\textsf{osc}\},\textsf{all}}^\perp \leq \mathfrak{fs}^\star$. Let us see that $\mathfrak{fs}=\mathfrak{fs}^\star$. Note that $\mathfrak{fs}\leq\mathfrak{fs}^\star$ by definition, so it is enough to show that $\mathfrak{fs}^\star \leq \max\{\mathfrak{s},\mathfrak{d}\}$. This is quite similar to \cite[Proposition 2.1]{Kamburelis1996-KAMS}, we just point out that the sets can be chosen without density. Let $\mathcal{S}$ be a splitting family such that $|\mathcal{S}|=\mathfrak{s}$, and let $\{f_\alpha : \alpha<\mathfrak{b}\}$ be an unbounded family. Without loss of generality, we can ask that every set in $\mathcal{S}$ is infinite-coinfinite, and $2^{f_\alpha(n)}\leq f_\alpha(n+1)$ for all $n\in\omega$, and for all $\alpha<\mathfrak{b}$. For $A\in\mathcal{S}$ and $\alpha<\mathfrak{b}$ define
\[
F_{A,\alpha}=\bigcup \{[f_\alpha(n),f_\alpha(n+1)):n\in A]
\]
Note that $\frac{|F_{A,\alpha} \cap f_\alpha(n+1) |}{f_\alpha(n+1)}\geq 1- f_\alpha(n)/2^{f_\alpha(n)}$ if $n\in A$, and $\frac{|F_{A,\alpha} \cap f_\alpha(n+1) |}{f_\alpha(n+1)}\leq f_\alpha(n)/2^{f_\alpha(n)}$ if $n\not\in A$. Then $F_{A,\alpha}$ oscillates for all $A\in\mathcal{S}$ and $\alpha<\mathfrak{b}$. In \cite[Proposition 2.1]{Kamburelis1996-KAMS}, it was proved that  every block sequence is split by some $F_{A,\alpha}$.
\end{proof}

We will finish this section showing that the covering of the strong measure zero ideal is an upper bound of $\mathfrak{r}_{(0,1)}$ and $\mathfrak{dd}_{(0,1),\textsf{all}}$. Let us recall the definition of a strong measure zero set.

\begin{defn}
A set $X\subseteq 2^\omega$ has strong measure zero set if for every $f\in \omega$ there exists $g\in(2^{<\omega})^\omega$ such that $g(n)\in 2^{f(n)}$ for all $n\in \omega$ and for all $x\in X$ there is $n\in \omega$ such that $x\upharpoonright f(n) =g(n)$.
\end{defn}

There are many equivalent definitions of a strong measure zero set. We are going to use the following characterization.

\begin{lemma}\label{characterizacin strong measure zero sets}
A set $X\subseteq 2^\omega$ has strong measure zero if and only if for every family of disjoint intervals $\{I_n:n\in \omega\}$ covering $\omega$, there exists $z \in 2^\omega$ such that
\[
\forall x \in X \exists^\infty n (x\upharpoonright I_n = z\upharpoonright I_n)\]
\end{lemma}
\begin{proof}
See \cite[Lemma 8.1.13.]{Bartoszynski1995SetTO}.  
\end{proof}

\begin{lemma}\label{smz and splitting}
Let $\mathcal{S}$ be a family of infinite subsets of $\omega$, and consider the family $\mathcal{S}^\prime \subseteq 2^\omega$ of the characteristic functions of the elements of $\mathcal{S}$. The following holds:
\begin{center}
if $\mathcal{S}^\prime$ has strong measure zero, then there exists $Z\in[\omega]^\omega$ such that for all $S\in \mathcal{S}$ we have $\overline{d}_Z(S)=1$ and $\underline{d}_Z(S)=0$.    
\end{center}

\end{lemma}
\begin{proof}
 Consider an interval partition $\{I_n\}_{n\in\omega}$ of $\omega$, such that $|I_0|\geq 2$ and $|I_{n}|>2^{n+1}|\bigcup_{i<n}I_i|$. By \ref{characterizacin strong measure zero sets}, there exists $z \in 2^\omega$ such that
\[
\forall s \in \mathcal{S}^\prime \exists^\infty n (s\upharpoonright I_n = z\upharpoonright I_n)
\]
Consider the set $Z\in [\omega]^\omega$ such that for all $n\in \omega$ we have $Z \cap I_n$ is the largest of the sets $\{i \in I_n \, | \, z(i)=0\}$ and $\{i \in I_n \, | \, z(i)=1\}$ (if both have the same size, choose any set). Fix $S \in \mathcal{S}$ with characteristic function $s$. For all $n\in \omega$ such that $(s\upharpoonright I_n = z\upharpoonright I_n)$ we have:
\begin{itemize}
    \item If $Z \cap I_n$ is $\{i \in I_n \, | \, z(i)=0\}$ then
    \[
    \frac{|S\cap Z \cap (\max I_n+1)|}{|Z \cap (\max I_n+1)|}\leq \frac{|S\cap Z \cap \bigcup_{j<n}I_j |}{|Z \cap I_n|} \leq \frac{2|\bigcup_{i<n}I_i|}{|I_n|} \leq \frac{2|\bigcup_{i<n}I_i|}{2^{n+1}|\bigcup_{i<n}I_i|} = \frac{1}{2^n}
    \]
    where the first inequality holds because $S$ is disjoint of $Z$ in $I_n$, and the second inequality holds because $|I_n|\leq 2|Z\cap I_n|$.
    \item If $Z \cap I_n$ is $\{i \in I_n \, | \, z(i)=1\}$ then
    \[
    \frac{|S\cap Z \cap (\max I_n+1)|}{|Z \cap (\max I_n+1)|}\geq \frac{|Z \cap I_n |}{|\bigcup_{j<n}I_j|+|Z\cap I_n|} \geq \frac{|I_n|}{2(|\bigcup_{j<n}I_j|+\frac{|I_n|}{2})} \geq \frac{|I_n|}{\frac{|I_n|}{2^n}+|I_n|}= \frac{1}{\frac{1}{2^n}+1}
    \]
    where the first inequality holds because $S$ coincides with $Z$ in $I_n$, and the second inequality holds because $|I_n|\leq 2|Z\cap I_n|$. 
    
\end{itemize}

Since the above holds for infinitely many $n\in\omega$, we conclude that $\overline{d}_Z(S)=1$ and $\underline{d}_Z(S)=0$.    
\end{proof}

\begin{cor}
For all nonempty $X\subseteq (0,1)$ we have non$(\mathcal{SN}) \leq \mathfrak{s}_X $ and $\mathfrak{r}_X \leq$ cov$(\mathcal{SN})$.
\end{cor} 
\begin{proof}
 Let $\mathcal{S}_X$ be a family of infinite-coinfinite subsets of $\omega$ witness of $\mathfrak{s}_X$. Note that this set does not have strong measure zero because \ref{smz and splitting} and $0, 1$ and $\textsf{osc}$ do not belong to $X$. Therefore, non$(\mathcal{SN}) \leq \mathfrak{s}_X$. To see the other inequality, consider a collection $\{N_\alpha\}_{\alpha<\kappa}$ of strong measure zero sets such that $\bigcup_{\alpha<\kappa}N_\alpha=[\omega]^\omega$. For each $N_\alpha$, consider the set $Z_\alpha$ given by \ref{smz and splitting} and define the family $\mathcal{R}_X:=\{Z_\alpha \,|\, \alpha<\kappa\}$. Therefore, given $S\in[\omega]^\omega$ there is $\alpha<\kappa$ such that $S\in N_\alpha$, so $\overline{d}_{Z_\alpha}(S)=1$ and $\underline{d}_{Z_\alpha}(S)=0$. Since $0, 1$ and $\textsf{osc}$ do not belong to $X$, we have that $\mathfrak{r}_X \leq$ cov$(\mathcal{SN})$. 
\end{proof}

\section{Relative density number} \label{section 2}

In this section, we prove that $\mathfrak{dd}_{(0,1),\{0,1\}}^{\textsf{rel}}$ and $\mathfrak{dd}_{\textsf{all},\textsf{all}}^{\textsf{rel}}$ are consistently different from $\mathfrak{c}$. We present a forcing notion $\mathbb{P}$ that adds a permutation $\pi$ of $\mathbb{N}$ and forces $d_{\pi(B)}(\pi(A))\in \{0,1\}$ for all pairs of sets $A\subseteq B \subseteq \omega$ in the ground model. This proof is similar to the proof of Con$(\mathfrak{rr}_i < \mathfrak{c})$ presented in \cite[Section 9]{blass2019rearrangementnumber}. As they do with the conditional convergent series, we are going to separate the collection of pairs of sets $A\subseteq B \subseteq \omega$ into two categories: one that will have relative density zero under the permutation, and another that will have relative density one. To do that, fix an enumeration $\langle A^\beta, B^\beta \, | \, \beta < \mathfrak{c} \rangle$  of the collection of pairs of sets $A\subseteq B \subseteq \omega$ such that $A$ and $B\setminus A$ are infinite. We construct an equivalence relation on the set $\mathfrak{c}$ of indices, and for each equivalence class we use its smallest ordinal number as a standard representative. In the notation of the following lemma, $\mathcal{A}$ will be the set of these representatives, and $\zeta$ will be the function sending each ordinal to the representative of its equivalence class.

\begin{lemma}\label{matrix}
Assume MA($\sigma$-centered). There exist a set $\mathcal{A} \subseteq \mathfrak{c}$, a function $\zeta: \mathfrak{c} \rightarrow \mathcal{A}$, and a matrix of sets $\langle X_\alpha^\beta : \alpha \in \mathcal{A}$ and $\alpha\leq \beta < \mathfrak{c} \rangle$ with the following properties for all $\beta < \mathfrak{c}$:
\begin{enumerate}
    \item $\zeta(\beta)\leq \beta$ with equality if and only if $\beta \in \mathcal{A}$.
    \item If $\alpha \in \mathcal{A}$ and $\alpha \leq \beta \leq \beta^\prime$, then $X_\alpha^{\beta^\prime} \subseteq^* X_\alpha^\beta$.
    \item The sets $X_\alpha^\beta$ for $\alpha \in \mathcal{A} \cap (\beta + 1)$ are almost disjoint. 
    \item $X_{\zeta(\beta)}^\beta$ is a subset of $A^\beta$ or of $B^\beta \setminus A^\beta$.
    \item If $\alpha\in \mathcal{A}$ and $\alpha< \zeta(\beta)$, then $X_\alpha^{\beta} \cap B^\beta =^* \emptyset$.
\end{enumerate}
\end{lemma}
\begin{proof}
We construct the matrix by recursion on ordinals $\beta < \mathfrak{c}$. In each step $\beta < \mathfrak{c}$, we are going to construct a set $\mathcal{A}_\beta \subseteq \beta + 1$. $\mathcal{A}$ will be $\bigcup_{\beta<\mathfrak{c}} \mathcal{A}_\beta$.

\begin{description}
    \item[Case $\beta=0$.] Let $\mathcal{A}_0:=\{0\}$ and set $\zeta(0)=0$. For $X_0^0$ choose either $A^0$ or $B^0 \setminus A^0$.
    \item[Case $\beta+1$.] Suppose we already constructed $\mathcal{A}_\beta$, the collection $\{X_\alpha^\beta \, : \, \alpha \in \mathcal{A}_\beta \}$, and defined $\zeta(\gamma)$ for $\gamma\leq \beta$. We proceed by recursion on $\alpha \in \mathcal{A}_\beta$. At step $\alpha$ of this recursion, we consider the set $X_\alpha^\beta \cap (\omega\setminus B^{\beta+1})$. If this set is infinite, then we declare it to be $X_\alpha^{\beta+1}$. Then we continue to the next value of $\alpha$. 
    
    If $X_\alpha^\beta \cap \omega\setminus B^{\beta+1}$ is finite, we stop the recursion on $\alpha$, we define $\zeta(\beta +1 )=\alpha$ and $\mathcal{A}_{\beta+1}=\mathcal{A}_\beta$. Since in this case $X_\alpha^\beta\subseteq^* B^{\beta+1}$, we set $X_\alpha^{\beta+1}$ either $X_\alpha^\beta \cap A^{\beta+1}$ or $X_\alpha^\beta \cap (B^{\beta+1}\setminus A^{\beta+1})$, whenever this set is infinite. For all $\gamma\in \mathcal{A}_\beta$ such that $\alpha<\gamma\leq \beta$ we define $X_\gamma^{\beta+1}:=X_\gamma^\beta$.  

    If the recursion on $\alpha$ does not stop, we define $\mathcal{A}_{\beta+1}:=\mathcal{A}_\beta \cup \{\beta+1\}$, $\zeta(\beta+1):=\beta+1$ and $X_{\beta+1}^{\beta+1}:=A^{\beta+1}$. Note that $X_\alpha^{\beta+1} \in [\omega\setminus B^{\beta+1}]^\omega$ for $\alpha\leq \beta$, so this set satisfies the fifth clause.

    \item[Case $\beta$ is a limit ordinal.]  For each $\alpha \in \mathcal{A}_\beta$ construct a pseudointersection $Y_\alpha$ of the collection $\{X_\alpha^\gamma \, : \, \gamma <\beta\}$. This is possible due to MA$(\sigma-$centered$)$. Now, proceed exactly as in the successor case, using $Y_\alpha$ in place of $X_\alpha^\beta$. 
\end{description}
\end{proof}

\begin{defn}
We say that an ordinal $\beta< \mathfrak{c}$ is a 1-ordinal if $X_{\zeta(\beta)}^\beta$ is a subset of $A^\beta$. Otherwise, we say that $\beta$ is a $0$-ordinal. We write $R(\beta)$ for $A^\beta$ when $\beta$ is a 1-ordinal and for $B^\beta \setminus A^\beta$ when $\beta$ is a 0-ordinal.    
\end{defn}

Let $\mathbb{P}$ be the following forcing. A condition is a triple $(s,F,k)$ such that
\begin{itemize}
    \item $s$ is an injective function with ran$(s)=n$ for some $n\in \mathbb{N}$,
    \item $F$ is a finite subset of $\mathfrak{c}$,
    \item $k\in \mathbb{N}$, and
    \item For all $\beta \in F$ we have
    \[
    \frac{|s[R(\beta)] \cap n |}{|s[B^\beta] \cap n |} > 1- \frac{1}{k} 
    \]
\end{itemize}
A condition $(s^\prime, F^\prime, k^\prime)$ extends $(s,F,k)$ if
\begin{itemize}
    \item $s^\prime \supseteq s$,
    \item $F^\prime \supseteq F$, 
    \item $k^\prime\geq k$, and
    \item for all $j \in $ ran$(s^\prime)\setminus$ ran$(s)$ and all $\beta \in F$ we have
    \[
    \frac{|s^\prime[R(\beta)] \cap j |}{|s^\prime[B^\beta] \cap j |} > 1- \frac{1}{k} 
    \]
\end{itemize}

\begin{lemma}\label{density of conditions}
Assume MA$(\sigma$-centered).
\begin{enumerate}
    \item $\mathbb{P}$ is $\sigma$-centered.
    \item  For every $l,a\in\mathbb{N}$, every condition $(s,F,k)$ has an extension $(s^\prime, F^\prime, k^\prime)$ with $k^\prime\geq l$ and $a \in$ dom$(s^\prime)$.
    \item For every $\gamma \in \mathfrak{c}$, every condition $(s,F,k)$ has an extension $(s^\prime, F^\prime, k^\prime)$ with $\gamma \in F^\prime$.
\end{enumerate}
\end{lemma}

\begin{proof}
\begin{enumerate}
    \item Note that for all $s,k$ the set $D_{s,k}:=\{(s,F,k) \, | \, F$ is a finite subset of $\mathfrak{c}\}$ is centered. 
    \item It suffices to prove the result with $l=k+1$. Let $(s,F,k)$ be a condition with ran$(s)=n$. We will construct an injective function $s^\prime \supseteq s$ such that $a \in$ dom$(s^\prime)$ and $(s^\prime, F, k+1)$ is an extension of $(s,F,k)$. Suppose that $F\neq \emptyset$ and $a \not\in$ dom$(s)$ (otherwise the construction is trivial). Let $\delta$ be the largest element of $F$. Enumerate $\zeta[F]$ in increasing order as $\zeta_0 < \zeta_1 <...< \zeta_{m-1}$. By backward recursion on $j<m$, we are going to find finite sets $Z_j \subseteq \mathbb{N}$ such that
    \begin{itemize}
        \item the sets $Z_j$ are disjoint from each other and from dom$(s)$,
        \item $a\not\in Z_{j}$,
        \item $Z_j \subseteq X_{\zeta_j}^\delta \cap \bigcap \{ R(\beta) : \beta \in F \text{ and } \zeta(\beta)=\zeta_j \}$,
        \item for all $\beta \in F$ with $\zeta(\beta)>\zeta_j$ we have $Z_j \cap B^\beta = \emptyset$.
        \item for all $\beta \in F$ with $\zeta(\beta)=\zeta_j$ we have
        \[
         \frac{|s[R(\beta)] \cap n |+|Z_j|}{|s[B^\beta] \cap n |+\sum_{j^\prime \geq j}|Z_{j^\prime}|+1} > 1- \frac{1}{k+1}
        \]   
    \end{itemize}
    Note that we can construct the collection $\{Z_j : j<m\}$ because for $j<m$ the sets
        \[
        X_{\zeta_j}^\delta \cap \bigcap \{ R(\beta) : \beta \in F \text{ and } \zeta(\beta)=\zeta_j \}
        \]
        are infinite and almost disjoint thanks to clauses (2), (3) and (4) of lemma \ref{matrix}. Therefore, fix $j<m$ and suppose that the sets $Z_{j^\prime}$ are already constructed for $j<j^\prime<m$. By the above, we can find a finite set $Z_j$ disjoint from $Z_{j^\prime}$ for $j<j^\prime<m$, dom$(s)$ and $\{a\}$. Besides, by clause (5) of lemma \ref{matrix}, we can also ask that for all $\beta \in F$ with $\zeta(\beta)>\zeta_j$ we have $Z_j \cap B^\beta = \emptyset$. For the last requirement, note that 
        \[
        \lim_{l \to \infty} \frac{|s[R(\beta)] \cap n |+l}{|s[B^\beta] \cap n |+\sum_{j^\prime >j }|Z_{j^\prime}|+1+l} = 1,
        \]
        so choose $Z_j$ big enough to satisfy the last inequality. 

        Now, define $s^\prime$ as follows: the domain of $s^\prime$ is dom$(s) \cup \bigcup_{j<m} Z_j \cup \{a\}$ and
        \[
        s^\prime(i)=\left\{
        \begin{array}{ll}
           s(i)  & \text{ if $i\in$ dom$(s)$} \\
           \text{ran}(s) + \sum_{j^\prime<j}|Z_{j^\prime}|+t  & \text{ if $i$ is the $t^{th}$ element of $Z_j$} \\
           \text{ran}(s) + \sum_{j^\prime< m}|Z_{j^\prime}|  & \text{ $i=a$ }
        \end{array}\right.
        \]
        Let us see that $(s^\prime, F, k+1)$ is a condition of the forcing $\mathbb{P}$. Denote $\text{ran}(s) + \sum_{j^\prime< m}|Z_{j^\prime}|+1$ by $d$. Note that $d$ is the range of $s^\prime$. Let $\beta \in F$ be such that $\zeta(\beta)=\zeta_j$, then
        \[
        \begin{array}{cl}
        \displaystyle
            \frac{|s^\prime[R(\beta)] \cap d |}{|s^\prime[B^\beta] \cap d |} & 
             \displaystyle= \frac{|s[R(\beta)] \cap n |+\sum_{j^\prime<m}|R(\beta) \cap Z_{j^\prime}|+|R(\beta) \cap \{a\}|}{|s[B^\beta] \cap n |+\sum_{j^\prime<m}|B^\beta \cap Z_{j^\prime}|+|B^\beta \cap \{a\}|}   \\ \\
             &\displaystyle\geq \frac{|s[R(\beta)] \cap n |+|R(\beta) \cap Z_j|}{|s[B^\beta] \cap n |+\sum_{j^\prime\geq j}|B^\beta \cap Z_{j^\prime}|+|B^\beta \cap \{a\}|} \\ \\ 
             &\displaystyle\geq \frac{|s[R(\beta)] \cap n |+|Z_j|}{|s[B^\beta] \cap n |+\sum_{j^\prime\geq j}| Z_{j^\prime}|+1}  \\ \\
             &\displaystyle > 1- \frac{1}{k+1}
        \end{array} 
        \]
        where the first equality holds because the sets $Z_j$, dom$(s)$ and $\{a\}$ are disjoint from each other, the first inequality holds because there are fewer terms in the numerator and for all $j^\prime<j$ we have $Z_{j^\prime} \cap B^\beta = \emptyset$, and the second inequality holds because $Z_j\subseteq R(\beta)$ and the denominator is bigger. The last inequality holds by the last condition of the construction of $Z_j$.

        Now, let us see that $(s^\prime, F, k+1)$ is an extension of $(s, F, k)$. Since $s^\prime$ is an extension of $s$, we only need to prove that for all $l \in $ ran$(s^\prime)\setminus$ ran$(s)$ and all $\beta \in F$ we have
        \[
         \frac{|s^\prime[R(\beta)] \cap l |}{|s^\prime[B^\beta] \cap l |} > 1- \frac{1}{k} 
        \]
        Fix $l\in \mathbb{N}$ such that $n\leq l < d$, and $\beta \in F$ such that $\zeta(\beta)=\zeta_j$. Then, there is $h<m$, and $i\in Z_h$ such that $i$ is the $t^{th}$ element of $Z_h$ and $s^\prime(i)=l$. Denote the first $t$ elements of $Z_h$ by $Z_h^t$. Therefore 
        \[
        \frac{|s^\prime[R(\beta)] \cap l |}{|s^\prime[B^\beta] \cap l |} \displaystyle= \frac{|s[R(\beta)] \cap n |+\sum_{j^\prime<h}|R(\beta) \cap Z_{j^\prime}|+|R(\beta) \cap Z_h^t|}{|s[B^\beta] \cap n |+\sum_{j^\prime<h}|B^\beta \cap Z_{j^\prime}|+|B^\beta \cap Z_h^t| }.
        \]
        If $h<j$, we have that  $\frac{|s^\prime[R(\beta)] \cap l |}{|s^\prime[B^\beta] \cap l |}= \frac{|s[R(\beta)] \cap n |}{|s[B^\beta] \cap n |} > 1 - 1/k$ since $B^\beta \cap Z_{j^\prime} = \emptyset$ for $j^\prime<j$. If $h=j$, then $Z_h^t \subseteq R(\beta)$, so
        \[
        \frac{|s^\prime[R(\beta)] \cap l |}{|s^\prime[B^\beta] \cap l |} \displaystyle = \frac{|s[R(\beta)] \cap n |+|R(\beta) \cap Z_h^t|}{|s[B^\beta] \cap n |+|B^\beta \cap Z_h^t| }= \frac{|s[R(\beta)] \cap n |+t}{|s[B^\beta] \cap n |+t }>1-\frac{1}{k}.
        \]
        Finally, if $j<h$ we have 
        \[
        \frac{|s^\prime[R(\beta)] \cap l |}{|s^\prime[B^\beta] \cap l |} \displaystyle \geq \frac{|s[R(\beta)] \cap n |+|Z_j|}{|s[B^\beta] \cap n |+\sum_{j^\prime\geq j}| Z_{j^\prime}|+1} > 1 - \frac{1}{k+1}.
        \]

    \item Let $(s,F,k)$ be a condition and let $\gamma \in \mathfrak{c}$. We will construct an injective function $s^\prime \supseteq s$ such that $(s^\prime, F \cup \{\gamma\}, k)$ is an extension of $(s,F,k)$. Suppose that $\gamma \not\in F$ (otherwise the construction is trivial). Let $\delta$ be the largest element of $F\cup \{\gamma\}$. Enumerate the set $\{ \zeta \in \zeta[F] \, | \, \zeta < \zeta(\gamma)\}$ in increasing order as $\zeta_0 < \zeta_1 <...< \zeta_{m-1}$. Denote $\zeta(\gamma)$ as $\zeta_m$. Therefore $\zeta_0 < \zeta_1 <...< \zeta_{m-1}<\zeta_m$. By backward recursion on $j\leq m$, we are going to find finite sets $Z_j \subseteq \mathbb{N}$ such that
    \begin{itemize}
        \item the sets $Z_j$ are disjoint from each other and from dom$(s)$,
        \item $Z_j \subseteq X_{\zeta_j}^\delta \cap \bigcap \{ R(\beta) : \beta \in F \cup \{\gamma\} \text{ and } \zeta(\beta)=\zeta_j \}$,
        \item for all $\beta \in F \cup \{\gamma\}$ with $\zeta(\beta)>\zeta_j$ we have $Z_j \cap B^\beta = \emptyset$.
        \item for all $\beta \in F\cup \{\gamma\}$ with $\zeta(\beta)=\zeta_j$ we have
        \[
         \frac{|s[R(\beta)] \cap n |+|Z_j|}{|s[B^\beta] \cap n |+\sum_{j^\prime \geq j}|Z_{j^\prime}|} > 1- \frac{1}{k}
        \]
    \end{itemize}
    The same argument as in the proof of part (2) yields the existence
    of sets $Z_j$ satisfying these requirements. Define $s^\prime$ as follows: the domain of $s^\prime$ is dom$(s) \cup \bigcup_{j\leq m} Z_j $ and
        \[
        s^\prime(i)=\left\{
        \begin{array}{ll}
           s(i)  & \text{ if $i\in$ dom$(s)$} \\
           \text{ran}(s) + \sum_{j^\prime<j}|Z_{j^\prime}|+t  & \text{ if $i$ is the $t^{th}$ element of $Z_j$}
        \end{array}\right.
        \]
    Let us see that $(s^\prime, F \cup \{\gamma\}, k)$ is a condition of the forcing $\mathbb{P}$. Denote $\text{ran}(s) + \sum_{j^\prime\leq m}|Z_{j^\prime}|+1$ by $d$. Note that $d$ is the range of $s^\prime$. Let $\beta \in F \cup \{\gamma\}$. Then 
    \[
     \frac{|s^\prime[R(\beta)] \cap d |}{|s^\prime[B^\beta] \cap d |} = \frac{|s[R(\beta)] \cap n |+\sum_{j^\prime\leq m}|R(\beta) \cap Z_{j^\prime}|}{|s[B^\beta] \cap n |+\sum_{j^\prime\leq m}|B^\beta \cap Z_{j^\prime}|}
     \]
    If $\zeta(\beta)>\zeta(\gamma)$, we have  $\frac{|s^\prime[R(\beta)] \cap d |}{|s^\prime[B^\beta] \cap d |} = \frac{|s[R(\beta)] \cap n |}{|s[B^\beta] \cap n |} > 1- \frac{1}{k} $ since $B^\beta \cap Z_j^\prime= \emptyset$ for $j^\prime \leq m$. If $\zeta(\beta)=\zeta_j$ for some $j\leq m$  we have
    \[
        \begin{array}{cl}
        \displaystyle
            \frac{|s^\prime[R(\beta)] \cap d |}{|s^\prime[B^\beta] \cap d |} 
             &\displaystyle\geq \frac{|s[R(\beta)] \cap n |+|R(\beta) \cap Z_j|}{|s[B^\beta] \cap n |+\sum_{j^\prime\geq j}|B^\beta \cap Z_{j^\prime}|} \\ \\ 
             &\displaystyle\geq \frac{|s[R(\beta)] \cap n |+|Z_j|}{|s[B^\beta] \cap n |+\sum_{j^\prime\geq j}| Z_{j^\prime}|}  \\ \\
             &\displaystyle > 1- \frac{1}{k}
        \end{array} 
     \]
     where the first inequality holds because there are fewer terms in the numerator and for all $j^\prime<j$ we have $Z_{j^\prime} \cap B^\beta = \emptyset$, and the second inequality holds because $Z_j\subseteq R(\beta)$ and the denominator is bigger. The last inequality holds by the last condition of the construction of $Z_j$.

    The same argument as in the proof of part (2) proves that $(s^\prime, F \cup \{\gamma\}, k)$ is an extension of $(s,F,k)$.   
\end{enumerate}    
\end{proof}

\begin{cor}\label{effect of P}
Assume MA($\sigma$-centered), and let $G$ be a $\mathbb{P}$-generic filter. Let $\pi:=\bigcup \{s:(s,F,k)\in G\}$. Therefore, $\mathbb{P}$ forces that $\pi$ is a permutation of $\mathbb{N}$ and for all pairs of sets $A\subseteq B \subseteq \omega$ we have $d_{\pi(B)}(\pi(A))\in \{0,1\}$.
\end{cor}
\begin{proof}
Using a density argument, part (2) of Lemma \ref{density of conditions} ensures that $\pi$ is a permutation of $\mathbb{N}$. In addition, part (3) of Lemma \ref{density of conditions} and the definition of $\mathbb{P}$ guarantee that for all pairs of sets $A\subseteq B \subseteq \omega$ with $d_B(A)\in (0,1)$ we have $d_{\pi(B)}(\pi(A))\in \{0,1\}$.      
\end{proof}

\begin{theorem}\label{consistency1}
It is consistent with ZFC that $\mathfrak{dd}_{(0,1),\{0,1\}}^{\textsf{rel}} = \aleph_1 < \mathfrak{c}$.
\end{theorem}
\begin{proof}
Let $\mathbb{Q}$ be a two-step iteration forcing in which the first step forces MA($\sigma$-centered), and the second step is $\mathbb{P}$. Consider a model $V$ where $\mathfrak{c}>\aleph_1$, and let $\mathbb{Q}_{\aleph_1}$ be the finite support iteration of $\mathbb{Q}$ of length $\omega_1$. By \ref{effect of P}, we have $V^{\mathbb{Q}_{\aleph_1}}\models \mathfrak{dd}_{(0,1),\{0,1\}}^{\textsf{rel}} = \aleph_1 < \mathfrak{c}$ because, thanks to the countable chain condition, every pair of sets $A\subseteq B \subseteq \omega$ appears at some intermediate stage of the iteration. 
\end{proof}

The previous theorem has a relation with \cite[Question 51]{brech2024densitycardinals} since there is an analogy between $\mathfrak{dd}_{(0,1),\{0,1\}}^{\textsf{rel}}$ and $\mathfrak{rr}_i$. It is not known if these cardinals are equal. 

\begin{theorem}\label{consistnecy2}
It is consistent with ZFC that $\mathfrak{dd}_{\textsf{all},\textsf{all}}^{\textsf{rel}} = \aleph_1 < \mathfrak{c}$.   
\end{theorem}
\begin{proof}
Let $\mathbb{Q}$ be a three-step iteration forcing in which the first step forces MA($\sigma$-centered), the second step is $\mathbb{P}$, and the third step is Cohen forcing $\mathbb{C}$. Consider a model $V$ where $\mathfrak{c}>\aleph_1$, and let $\mathbb{Q}_{\aleph_1}$ be the finite support iteration of $\mathbb{Q}$ of length $\omega_1$. It is easy to check that the Cohen forcing adds a permutation $\sigma_\mathbb{C}$ such that $d_{\sigma_\mathbb{C}(B)}(\sigma_\mathbb{C}(A))= \textsf{osc}$ for every pair of sets $A\subseteq B \subseteq \omega$ with $A$ and $B \setminus A$ both infinite in the ground model. Therefore, using the previous mentioned and \ref{effect of P} it can be shown that $V^{\mathbb{Q}_{\aleph_1}}\models \mathfrak{dd}_{\textsf{all}, \textsf{all}}^{\textsf{rel}} = \aleph_1 < \mathfrak{c}$.
\end{proof}

If we pay attention to the details, we can get a stronger result. 

\begin{cor}
Con($\mathfrak{dd}_{(0,1),\{0,1\}}^{\textsf{rel}} <$ non$(\mathcal{N})$) and Con($\mathfrak{dd}_{\textsf{all}, \textsf{all}}^{\textsf{rel}}<$ non$(\mathcal{N})$)
\end{cor}
\begin{proof}
Note that the forcings used in \ref{consistency1} and \ref{consistnecy2} are $\sigma$-centered because they are a finite support iteration of length $< \mathfrak{c}^+$ of $\sigma$-centered forcing notions.
\end{proof}

\section{Open problems}

We conclude this paper with a list of questions we were unable to solve. We proved that $\mathfrak{r}_X=\mathfrak{dd}_{X,\textsf{all}}=$ non$(\mathcal{M})$ for all $X\subseteq[0,1]$ such that $0\in X$ or $1\in X$. However, we still do not know if there exists $X\subseteq$ \textsf{all} such that $\mathfrak{dd}_{X,\textsf{all}}$ is consistently different from $\mathfrak{r}_X$. 

\begin{question}
Is $\mathfrak{dd}_{X,\textsf{all}} < \mathfrak{r}_X$ consistent for some $X\subseteq$ \textsf{all}?    
\end{question}

We also proved $\mathfrak{dd}_{\mathbb{Q}\cap(0,1),\textsf{all}}=\mathfrak{dd}_{\{1/2\}, \textsf{all}}$. Besides, using a similar argument, it is easy to see that $\mathfrak{r}_{\mathbb{Q}\cap(0,1)}=\mathfrak{r}_{1/2}$ but we do not know if $\mathfrak{r}_{\mathbb{Q}\cap(0,1)}$ is equal to $\mathfrak{r}_{(0,1)}$, or if $\mathfrak{dd}_{\mathbb{Q}\cap(0,1),\textsf{all}}$ is equal to $\mathfrak{dd}_{(0,1),\textsf{all}}$.

\begin{question}
Is $\mathfrak{r}_{1/2}<\mathfrak{r}_{(0,1)}$ consistent? Is $\mathfrak{dd}_{\{1/2\},\textsf{all}} < \mathfrak{dd}_{(0,1),\textsf{all}}$ consistent?     
\end{question}

Our proof of the theorem \ref{equalityparameters} does not involve a Tukey connection, so the dual problem is still open.

\begin{question}
Is it consistent that $\mathfrak{dd}^\perp_{\{\rho\},\textsf{all}} \neq \mathfrak{dd}^\perp_{\{\tau\},\textsf{all}}$ for some $\rho, \tau \in (0,1)$?    
\end{question}

We could not find if there is a relation between $\mathfrak{r}_{\textsf{osc}}$ and cov$(\mathcal{N})$. In particular, we could not answer if the random forcing adds a set $X$ such that $d_A(X)$ oscillates for all infinite-coinfinite set $A$ in the ground model.

\begin{question}
Is $\mathfrak{r}_{\textsf{osc}}\geq$ cov$(\mathcal{N})$?    
\end{question}

This question is related to \cite[Question 37]{brech2024densitycardinals} since $\mathfrak{r}_{\textsf{osc}}\geq\mathfrak{dd}_{\{\textsf{osc}\},\textsf{all}}$ and $\mathfrak{dd}_{\{\textsf{osc}\},\textsf{all}}$ has similar behavior to $\mathfrak{r}_{\textsf{osc}}$. Concerning the relative density number, we proved that $\mathfrak{dd}_{(0,1),\{0,1\}}^{\textsf{rel}}$ is consistently different from the continuum, but the problem is still open for $\mathfrak{dd}_{(0,1),(0,1)}^{\textsf{rel}}$.

\begin{question}
Is  $\mathfrak{dd}_{(0,1),(0,1)}^{rel}<\mathfrak{c}$ consistent?    
\end{question}

This question is related to \cite[Question 51]{brech2024densitycardinals} because there is an analogy between $\mathfrak{dd}_{(0,1),(0,1)}^{rel}$ and the rearrangement number $\mathfrak{rr}_{f}$. Besides, in \cite[Section 8]{blass2019rearrangementnumber}, it was proved that $\mathfrak{rr}_{f}$ can be consistently different from the continuum.   

\hspace{2mm}

\noindent\textbf{Acknowledgments.} The author wishes to thank his supervisor, Jörg Brendle, for his valuable ideas and insightful comments, which contributed to the development of this paper.

\bibliographystyle{amsalpha}
\bibliography{main}

\newcommand{\etalchar}[1]{$^{#1}$}
\providecommand{\bysame}{\leavevmode\hbox to3em{\hrulefill}\thinspace}
\providecommand{\MR}{\relax\ifhmode\unskip\space\fi MR }
\providecommand{\MRhref}[2]{%
  \href{http://www.ams.org/mathscinet-getitem?mr=#1}{#2}
}
\providecommand{\href}[2]{#2}
\begin{thebibliography}{BHK{\etalchar{+}}23}

\bibitem[BBB{\etalchar{+}}20]{blass2019rearrangementnumber}
Andreas Blass, Jörg Brendle, Will Brian, Joel~David Hamkins, Michael Hardy, and Paul~B. Larson, \emph{The rearrangement number}, Transactions of the American Mathematical Society \textbf{373} (2020), 41--69.

\bibitem[BBT24]{brech2024densitycardinals}
Christina Brech, Jörg Brendle, and Márcio Telles, \emph{Density cardinals}, 2024.

\bibitem[BHK{\etalchar{+}}23]{Brendle_2023}
Jörg Brendle, Lorenz~J. Halbeisen, Lukas~Daniel Klausner, Marc Lischka, and Saharon Shelah, \emph{Halfway new cardinal characteristics}, Annals of Pure and Applied Logic \textbf{174} (2023), no.~9, 103303.

\bibitem[BJ95]{Bartoszynski1995SetTO}
Tomek Bartoszynski and Haim Judah, \emph{Set theory: On the structure of the real line}, 1995.

\bibitem[Bre98]{Brendle1998}
Jörg Brendle, \emph{Around splitting and reaping}, Commentationes Mathematicae Universitatis Carolinae \textbf{39} (1998), no.~2, 269--279.

\bibitem[FKL23]{FARKAS_KLAUSNER_LISCHKA_2023}
Barnabás Farkas, Lukas~Daniel Klausner, and Marc Lischka, \emph{More on halfway new cardinal characteristics}, The Journal of Symbolic Logic (2023), 1–16.

\bibitem[KW96]{Kamburelis1996-KAMS}
A.~Kamburelis and B.~Weglorz, \emph{Splittings}, Archive for Mathematical Logic \textbf{35} (1996), no.~4, 263--277.

\bibitem[Val24]{valderrama2024cardinalinvariantsrelateddensity}
David Valderrama, \emph{Cardinal invariants related to density}, 2024.

\end{thebibliography}

\hspace{2mm}

\noindent \text{\small David Valderrama,} \textsc{\small Universidad de Los Andes (Bogotá).} 

\noindent \textit{\small E-mail address:} \text{\small \href{d.valderramah@uniandes.edu.co}{d.valderramah@uniandes.edu.co}}

\end{document}